\documentclass[11pt]{amsart}
\usepackage{enumerate}
\newtheorem{lema}{Lemma}[section]
\newtheorem{theo}[lema]{Theorem}
\newtheorem{prop}[lema]{Proposition}
\newtheorem{coro}[lema]{Corollary}

\newtheorem{rema}[lema]{Remark}
\newtheorem{conj}[lema]{Conjecture}
\usepackage{graphicx}
\begin{document}
\title[]
{Polynomial differential equations with piecewise linear coefficients}
\author[]
{M. A. M. Alwash}
\address{Department of Mathematics, West Los Angeles College
\newline
9000 Overland Ave, Culver City, CA 90230-3519, USA}
\email{alwashm@wlac.edu}
\numberwithin{equation}{section}
\begin{abstract}
Cubic and quartic non-autonomous differential equations with continuous piecewise linear coefficients are considered.
 The main concern is to find the maximum possible multiplicity of periodic solutions. For many classes, we show that the multiplicity is the same when the coefficients are polynomial functions of degree $n$, or piecewise linear functions with $n$ segments.
\end{abstract}
\maketitle

\section{Introduction}
Consider the Abel differential equation
\begin{equation}
\dot{z}=\frac{dz}{dt}=A(t)\,z^{3}+B(t)\,z^{2},
\end{equation}
where $z$ is real and $A(t),B(t)$ are continuous functions.
Let $z(t,c)$ be the solution that satisfies the initial condition $z(0,c)=c$.
   A solution $\varphi$ is {\emph periodic} if it satisfies the boundary condition $\varphi(0)=\varphi(1)$. The equation has a \emph{center} at $z=0$ if there exists an open interval $I$ containing $0$ such
  $z(t,c)$ is periodic for all $c$ in $I$. The concept is related to the classical center problem of polynomial two-dimensional systems, see \cite{A5}. Several research articles were published in the last twenty five years to find conditions which are necessary and sufficient for the existence of a center, see \cite{Y1}. The displacement function $q$ is defined by
\[
q(c)=z(1,c)-c.
\]
Zeros of $q$ identify initial points of solutions of periodic solutions.
Note that $q$ is a holomorphic function defined on an open set containing the origin. The \emph{multiplicity} of a periodic solution $\varphi$ is that of $\varphi (0)$ as a zero of $q$. In the neighborhood of $z=0$, we can write
\begin{equation}
z(t,c)=\sum_{n=1}^{\infty}a_{n}(t)\,c^{n},
\end{equation}
for $0 \leq t \leq 1$; the $a_{n}(t)$ are continuous and satisfy the following initial conditions.
\[a_{1}(0)=1,a_{n}(0)=0,\,\,n>1.\]
The multiplicity of the origin is $k$, $k \geq 2$, if and only if
\[a_{i}(1)=0,i=2,\cdots,k-1,a_{k}(1) \neq 0.\]
Moreover, $z=0$ is stable when $a_{k}(1)<0$, and it is unstable when $a_{k}(1)>0$.
  The Abel differential equation has a center at the origin if and only if
\[a_{k}(1)=0, k\geq 2.\]
The functions $a_{n}(t)$ satisfy the equations
\[a_{1}(t)\equiv 1,\]
and
\begin{equation}
\dot{a}_{n}=A \,\sum_{i+j+k=n}^{}a_{i}\,a_{j}\,a_{k}+B \,\sum_{i+j=n}^{} a_{i}\,a_{j}.
\end{equation}
Formulae for the $a_{n}(t)$, in terms of $A(t)$ and $B(t)$, were derived in \cite{A5} for $2 \leq n \leq 8$.

Now, we consider the quartic differential equation
\begin{equation}
\dot{z}=\frac{dz}{dt}=z^{4}+A(t)z^{3}+B(t)z^{2},
\end{equation}
where $z$ is real and $A(t),B(t)$ are continuous functions. With the same definition of multiplicity, the formulae for $a_{n}(t)$ becomes:
\begin{equation}
\dot{a}_{n}=\sum_{i+j+k+l=n}^{} a_{i}\,a_{j}\,a_{k}\,a_{l}+A\, \sum_{i+j+k=n}^{}a_{i}\,a_{j}\,a_{k}+B\, \sum_{i+j=n}^{} a_{i}\,a_{j}.
\end{equation}
Formulae for the $a_{n}(t)$, in terms of $A(t)$ and $B(t)$, were derived in \cite{A4} for $2 \leq n \leq 8$.

The class of equations (1.4) has received some attention in the literature. The main concern is to estimate the number of periodic solutions. The qualitative behavior of the
solution curves depends entirely on the periodic solutions; see, for example, \cite{P1}. The problem was suggested by C. Pugh
as a version of Hilbert's sixteenth problem; it is listed as Problem 7 by Steve Smale in \cite{S1}. Equations of the form (1.3), have been studied in \cite{L1} and \cite{P1} using the methods of complex analysis and topological dynamics. The variable $z$ was assumed to be complex.
The reason is that periodic solutions cannot then
be destroyed by small perturbations of the right-hand side of the equation. Suppose that $\varphi$ is a periodic
solution of multiplicity $k$. This solution is counted as $k$ solutions.
By applying Rouche's theorem to the function $q$, for any sufficiently small
perturbations of the equation, there are precisely $k$ periodic solutions in a neighborhood of $\varphi$ (counting
multiplicity). On the other hand, upper bounds to the number of periodic solutions of equation (1.4) can be used as upper bounds to the
number of periodic solutions when $z$ is limited to be real-valued.
This is the reason that the coefficients are not allowed to be complex-valued.
The results presented in \cite{L1} could be used for equations with piecewise linear coefficients; the coefficients in \cite{L1} are only required to be continuous.

The equation (1.4) was considered in \cite{A2} and \cite{A4}. The main concern was the multiplicity of
$z=0$ when the coefficients are polynomial functions in $t$, and in $\cos{t}$ and $\sin{t}$. Equations with
at least $10$ real periodic solutions were constructed. These periodic solutions are bifurcated from a periodic solution of
multiplicity $10$. In this paper, we consider the case in which $A(t)$ and $B(t)$ are continuous piecewise linear functions.

T simplify the presentation, we introduce the functions\\
{\bf Definition.}
$\mu_{1}(m,n)=$ Maximum\,\{multiplicity of $z=0$ when $B(t)$ and $A(t)$ are polynomial functions of degree $m$ and $n$, respectively\}.\\
$\mu_{2}(m,n)=$
maximum\,\{multiplicity of $z=0$ when $B(t)$ and $A(t)$ are continuous piecewise linear functions, with $m$ and $n$ segments,  respectively\}.\\
For $\mu_{2}$ the segments are connected at $\frac{k}{n},\,k=1,2,\cdots\,n-1$.

In the next section, we consider the cubic equation. First, sufficient conditions for the existence of a center at $z=0$ are given. Then we prove the following result.
\begin{theo}
For the equation (1.1),
\begin{enumerate}
  \item $\mu_{1}(2,2)=\mu_{2}(2,2)=4$,
  \item $\mu_{1}(2,3)=\mu_{2}(2,3)=8$,
  \item $\mu_{1}(1,2)=\mu_{1}(1,3)=\mu_{2}(1,2)=\mu_{2}(1,3)=4,\\
  \mu_{1}(1,4)=\mu_{2}(1,4)=\mu_{1}(1,5)=\mu_{2}(1,5)=5,\\
      \mu_{1}(1,6)=\mu_{2}(1,6)=10,\,\mu_{1}(1,7)=\mu_{2}(1,7)=11 $.
\end{enumerate}
\end{theo}
The quartic equation is considered in Section 3. We prove our second result.
\begin{theo}
For the equation (1.4),
\begin{enumerate}
  \item $\mu_{1}(2,2)=\mu_{2}(2,2)=8$,
  \item $\mu_{1}(2,3)=\mu_{2}(2,3)=10$,
  \item $\mu_{1}(1,2)=\mu_{1}(1,3)=\mu_{2}(1,2)=\mu_{2}(1,3)=5,\\\mu_{1}(1,4)=\mu_{2}(1,4)=9,\,\mu_{1}(1,5)=\mu_{2}(1,5)=10$.
\end{enumerate}
\end{theo}
These results provide evidences for the following conjecture.
\begin{conj}
 For the equations (1.1) and (1.4), $\mu_{1}(m,n)=\mu_{2}(m,n)$, for all $m$ and $n$.
\end{conj}
\section{Cubic Equations}
The formula (1.3) is nonlinear with an increasing number of terms. In computing multiplicity, another linear formula is used.
To derive this linear formula, we use the expansion of the inverse Poincar\'{e} mapping
\begin{equation}
 c= \sum_{k=1}^{\infty} \frac{1}{k}V_{k}(t)z^{k},
 \end{equation}
where $V_{0}(1)=1$, and $V_{k}(1)=0,\,k>1$.
From the two expansions (1.2) and (2.1), we have
\[c= \sum_{k=1}^{\infty} \frac{1}{k}\,V_{k}\,[\sum_{n=1}^{\infty}a_{n}\,c^{n}]^{k}.\]
It follows from equating the coefficients of $c^{i}$ in both sides that if $a_{i}(1)=V_{i}(1)=0,\,1<i<k-1$, then
\[a_{k}(1)=-\frac{1}{k}V_{k}(1).\]
Next, differentiate $c= \sum_{k=1}^{\infty} \frac{1}{k}V_{k}(t)\,z^{k}$ with respect to $t$ and then substitute in (1.1); we obtain
 \[ 0= \sum_{1}^{\infty} (\frac{1}{k}V'_{k}\,z^{k}+V_{k}\,z^{k-1}z')= \sum_{1}^{\infty} (\frac{1}{k}V'_{k}\,z^{k}+V_{k}\,z^{k-1}(A\,z^{3}+B\,z^{2})).\]
From equating the coefficients of $z^{k}$ in both sides, we have
\begin{equation}
V_{1}(t)\equiv 1,\,V_{2}(t)=-2 \int_{0}^{t}\, B(s)\,ds,
\end{equation}
\begin{equation}
V_{k}(t)=-k \int_{0}^{t}\,[B(s)\, V_{k-1}(s)+A(s) \,V_{k-2}(s)]\, ds, k>2.
\end{equation}
This formula is linear and easier to implement than the formula for $a_{n}$. We summarize these remarks as follows.
\begin{prop}
Suppose that $V_{i}(t)$ are defined by the formulae (2.2). The solution $z=0$ of equation (1.1) is of multiplicity $k$ if and only if $V_{i}(1)=0$ for $2 \leq i \leq k-1$ and $V_{k}(1) \neq 0$. The solution $z=0$ is stable when $V_{k}(1) > 0$, and is unstable when $V_{k}(1) < 0$.
\end{prop}
The procedure of using the inverse Poincare map is classical, see  \cite{N1}. Similar formulae were obtained in \cite{A3} using Liapunov functions approach.

First, we present conditions on the functions $A(t)$ and $B(t)$ that imply $z=0$ is a center.

\begin{theo}
 Suppose that $A(t)$ and $B(t)$ are continuous functions. The condition
 \[z(\frac{1}{2}+t)=z(\frac{1}{2}-t)\]
 is satisfied by all solutions of equation (1.1) if and only if
 \[ A(\frac{1}{2}+t)=-A(\frac{1}{2}-t),\, B(\frac{1}{2}+t)=-B(\frac{1}{2}-t),\]
 for $0 \leq t \leq \frac{1}{2}$.
 \end{theo}
\begin{proof}
Suppose that $A(t)$ and $B(t)$ satisfy the above condition. Let $z(t)$ be a solution of (1.1) defined on the interval $[0,1]$. Consider the functions
\[z_{1}(t)=z(\frac{1}{2}+t), z_{2}(t)=z(\frac{1}{2}-t).\]
The functions are defined on the interval $[0,\frac{1}{2}]$.
Differentiate $z_{1}(t)$ and $z_{2}(t)$ and then substitute in the differential equation; this gives
\[\dot{z}_{1}(t)=A(\frac{1}{2}+t)\,z^{3}(\frac{1}{2}+t)+B(\frac{1}{2}+t)\,z^{2}(\frac{1}{2}+t)=
A(\frac{1}{2}+t)\,z_{1}^{3}(t)+B(\frac{1}{2}+t)\,z_{1}^{2}(t),\]
\[\dot{z}_{2}(t)=-A(\frac{1}{2}-t)\,z^{3}(\frac{1}{2}-t)-B(\frac{1}{2}-t)\,z^{2}(\frac{1}{2}-t)=A(\frac{1}{2}+t)\,z_{2}^{3}(t)+
B(\frac{1}{2}+t)\,z_{2}^{2}(t).\]
Hence, $z_{1}(t)$ and $z_{2}(t)$  are solutions of the differential equation
\[\dot{z}(t)=A(\frac{1}{2}+t)\,z^{3}(t)+B(\frac{1}{2}+t)\,z^{2}(t).\]
On the other hand, $z_{1}(0)=z_{2}(0)=z(\frac{1}{2})$. The uniqueness theorem implies that $z_{1}(t) \equiv z_{2}(t)$. Therefore,
\[z(\frac{1}{2}+t)=z(\frac{1}{2}-t).\]
In particular, $z(0)=z(1)$.

Conversely, assume that all solutions $z(t)$ starting in a neighborhood of the origin satisfy the condition $z(\frac{1}{2}+t)=z(\frac{1}{2}-t)$ for $0 \leq t \leq \frac{1}{2}$. Now, we differentiate both sides and substitute in the equation to obtain
\[ z(\frac{1}{2}+t)[A(\frac{1}{2}+t)+A(\frac{1}{2}-t)] + B(\frac{1}{2}+t)+B(\frac{1}{2}-t)=0\]
for all small $z$. Therefore,
\[ A(\frac{1}{2}+t)=-A(\frac{1}{2}-t),\,B(\frac{1}{2}+t)=-B(\frac{1}{2}-t).\]
\end{proof}
If $A(t)$ and $B(t)$ satisfy the condition in Theorem 2.2 then
\[A_{1}(t)=A(t-\frac{1}{2}), B_{1}(t)=B(t-\frac{1}{2})\]
are odd functions. It follows from the theory of Fourier series that these functions are of the form $\sin(4\pi t)\, g(\cos(4 \pi t))$, where $g$ is a continuous function. Therefore, $A(t)$  and $B(t)$ satisfy the composition condition.
Recall that $A(t)$ and $B(t)$ satisfy the composition condition if $A(t)=s'(t)A_{1}(s(t))$ and $B(t)=s'(t)B_{1}(s(t))$, where $s(t)$ is a periodic function and $A_{1}$ and $B_{1}$ continuous functions. We refer the reader to \cite{A5} for more details.
\begin{lema}\cite{A5}
If $A(t)$ and $B(t)$ satisfy the composition condition then $z=0$ is a center for equation (1.1).
\end{lema}
Let $t_{i}=\frac{i}{n}$ for $i=0,1,2,\cdots,n$. If $f(t)$ is a continuous piecewise linear function defined on the interval $[0,1]$, then $f(t)$ can be written in the form
\[
f(t)=m_{1}t+b+\frac{1}{2} \sum_{i=2}^{i=n}(m_{i}-m_{i-1})(t-t_{i-1}+\mid t-t_{i-1} \mid).
\]
If $t_{k-1}\leq t \leq t_{k}$ then $t-t_{i}+|t-t_{i}|=0$ for $i > k-1$, and $t-t_{i}-|t-t_{i}|=2t-2t_{i}$ for $i \leq k-1$. Therefore, in each subinterval $[t_{i-1},t_{i}]$, the formula becomes
\[
f(t)=m_{k}t+b+\frac{1}{n} [m_{1}+m_{2}+\cdots+m_{k-1}-(k-1)m_{k}].
\]
The slope of this line segment is $m_{i}$
\begin{coro}
 Suppose that $A(t)$ and $B(t)$ are piecewise linear continuous functions. Let $m_{k}$ and $n_{k}$ be the slopes of the line segments
 $t_{k-1}<t<t_{k},k=1,2,\cdots,n$ of $A(t)$ and $B(t)$, respectively. If $A(\frac{1}{2})=B(\frac{1}{2})=0$, $m_{k}=m_{n-k}$ and $n_{k}=n_{n-k}$ for $k=1,2,\cdots,n$ then equation (1.1) has a center at $z=0$.
\end{coro}
\begin{proof}
We show that $A(t)$ satisfies the conditions in Theorem 2.2.
For a given $t$, let $S_{1},S_{2},\cdots,S_{k}$ be the segments that contain the interval $\frac{1}{2}\leq t \leq \frac{1}{2}+t$. This implies that
\[A(\frac{1}{2}+t)= s_{1}\frac{1}{2n}+s_{2}\frac{1}{n}+s_{3}\frac{1}{n}+\cdots+s_{k-1}\frac{1}{n}+ s_{k}(\frac{1}{2}+t-\frac{2k-3}{2n}),\]
where, $s_{i}$ is the slope of the line segment $S_{i}$.
Simplifying the right hand side gives
\[A(\frac{1}{2}+t)=\frac{(s_{1}+2s_{2}+\cdots+2s_{k-1}+s_{k}(n+2nt-2k+3))}{2n},\]
On the other hand,
\[A(\frac{1}{2}-t)= -s_{1}\frac{1}{2n}-s_{2}\frac{1}{n}-s_{3}\frac{1}{n}-\cdots-s_{k-1}\frac{1}{n}- s_{k}(\frac{1}{2}-t-\frac{2k-3}{2n}),\]
and hence,
\[A(\frac{1}{2}-t)=\frac{-(s_{1}+2s_{2}+\cdots+2s_{k-1}+s_{k}(n+2nt-2k+3))}{2n}.\]
The same argument is applied to $B(t)$. Hence, $A(t)$ and $B(t)$ satisfy the conditions in Theorem 2.2.
\end{proof}
\begin{rema}
\begin{enumerate}
\item
It follows from Corollary 2.3 that if $z=0$ is a center with respect to the interval $[0,1]$, then $z=0$ is a center with respect to any interval of the form $[\frac{1}{2}-r,\frac{1}{2}+r]$ with $0 \leq r \leq \frac{1}{2}$.
\item
Suppose that $f(t)$ satisfy the conditions in Corollary 2.2. If the number of segments is even then the two middle segments have the same slope and hence can be considered of one segment with length $\frac{2}{n}$ and we call it the middle segment. To understand the shape of $f(t)$, we start from middle segment and add two parallel segments, one at each side. We continue in this process until we cover the interval $[0,1]$. The only condition is that the point $(\frac{1}{2},0)$ is on the middle segment.
\end{enumerate}
\end{rema}
\medskip
\noindent
\textbf{Notation.} Let $\eta_{k}$ denotes $V_{k}(1)$ modulo the ideal generated by $\langle V_{2}(1),V_{3}(1),\cdots,V_{k-1} \rangle$. The multiplicity of the solution $z=0$ is $k$ when $\eta_{2}=\eta_{3}=\cdots=\eta_{k-1}=0,$ and $\eta_{k} \neq 0$.  We use the theory of Gr\"{o}bner bases to simply the base of an ideal. The Gr\"{o}bner basis of the ideal generated by $\langle \eta_{2},\eta_{3},\cdots,\eta_{k} \rangle$ is denoted by $G_{k}$. The Computer Algebra System Maple is used in computing Gr\"{o}bner bases; see \cite{D1}.
\medskip
Now, we prove Theorem 1.1. In the case that the coefficients are polynomial functions, the results were proved in \cite{A1}, \cite{A5} and \cite{B1}. We present the proofs for completeness.
\begin{proof}
(\emph {Theorem 1.1})
We compute the Gr\"{o}bner bases $G_{2},G_{3},\cdots,G_{k}$, such that vanishing all polynomials in $G_{k}$ implies that the origin is a center.
The existence of $G_{k}$ follows from Hilbert's finiteness theorem. The basis $G_{k}$ is the basis of the center ideal. This ideal is called the Bautin ideal; see, for example, \cite{Y1}. To show that $z=0$ is a center, we need sufficient and necessary conditions for a center.

\begin{enumerate}
  \item
  Let
  \[ B(t)=a+2\,bt+3\,c{t}^{2},\, A(t)=d+2\,et+3\,f{t}^{2}.\]
  The Gr\"{o}bner basis is given by
\[G_{4}=\langle ec-fb,a+b+c,f+e+d \rangle.\]
These three conditions imply that $B(t)= \frac{b}{e}\,A(t)$. Lemma 2.3 implies that $z=0$ is a center. With the notation, $s(t)$ is the definite integral of $A(t)$. Therefore, $\mu_{1}(2,2)=4$.
In the case of piecewise linear coefficients, we let
\[ B(t)=at+b+\frac{1}{2}\, \left( c-a \right)  \left( t-\frac{1}{2}+ \left| t-\frac{1}{2} \right|
 \right),\]
\[ A(t)=dt+e+\frac{1}{2}\, \left( f-d \right)  \left( t-\frac{1}{2}+ \left| t-\frac{1}{2} \right|
 \right).\]
We obtain a similar Gr\"{o}bner basis
\[G_{4}=\langle ce-fb,8\,b+c+3\,a,f+3\,d+8\,e \rangle \]
These conditions imply that $B(t)=\frac{b}{e}\,A(t)$. Same argument used above implies that $\mu_{2}(2,2)=4$.
  \item
Let
  \[ B(t)=a+2\,bt+3\,c{t}^{2},\,A(t)=d+2\,et+3\,f{t}^{2}+4\,g{t}^{3}\]
The basis is given by
\begin{gather*}
G_{7}=\langle a+b+c,g+f+e+d,1287\,gfc+1482\,{g}^{2}c+858\,gec-2\,g{b}^{2}c,\\-7\,ec+7
\,fb+14\,gb+9\,gc,
39\,{g}^{2}bc-7\,gc{b}^{3},3\,g{c}^{2}+2\,gbc,-3\,gc
{b}^{3}+13\,egbc \rangle.
\end{gather*}
The case $g=0$ is considered in part (1) and the case $c=0$ is considered in the part (3). We assume that $gc \neq 0$. The basis becomes
\[G_{7}=\langle 3\,a+b,7\,f-9\,h,2\,h+g,3\,c+2\,b,7\,e+2\,h,-39\,h+7\,{b}^{2} \rangle.\]
and $\eta_{8}$ is given by
\[\eta_{8}=-{\frac {17}{10405395}}\,c{g}^{3}+{\frac {4}{1486485}}\,ce{g}^{2}-{
\frac {1}{945945}}\,cg{e}^{2}+{\frac {41936}{15335981015355}}\,cg{b}^{
4}.\]
Adding $\eta_{8}$ to the basis gives
\[G_{8}=\langle 1\rangle. \]
Hence $\mu_{1}(2,3)=8$.

Next, let
\[ B(t)=at+b+\frac{1}{2}\, \left( c-a \right)  \left( t-\frac{1}{2}+ \left| t-\frac{1}{2} \right|
 \right) \]
\[A(t)=dt+e+\frac{1}{2}\, \left( f-d \right)  \left( t-\frac{1}{3}+ \left| t-\frac{1}{3} \right|
 \right) +\frac{1}{2}\, \left( g-f \right)  \left( t-\frac{2}{3}+ \left| t-\frac{2}{3}
 \right|  \right).\]
The case that $c-a=0$ is considered in the part (3). When $c-a \neq 0$, we make the change of variables
\[ z \mapsto \frac{1}{c-a} \,z. \]
This transformation reduces the equation into a similar one but $c-a$ is replaced by $1$. It should be mentioned that this transformation does not change the multiplicity. Hence, we let $c-a=1$. The basis is given by
\[G_{8}=\langle 4\,a+8\,b+1,d,e,4\,c-3+8\,b,f,g \rangle .\]
These conditions imply that $A(t)\equiv 0$, and hence the origin is a center.
\item
It follows from Proposition 2.1 that if the multiplicity is greater than $2$ then $\int_{0}^{1}\,B(t)\,dt=0$. If $B(t)$ is a linear function then  the multiplicity is greater than $2$ when $B(t)= u\,(2\,t-1)$, where $u$ is a constant. With the transformation
\[ z \mapsto \frac{1}{u} \,z \]
we can assume that $u=1$. We write the polynomials in the form
  \[B(t)=2\,t-1,\]
  \[ A(t)=
a+b \left( {t}^{2}-t \right) +c \left( {t}^{2}-t \right) ^{2}+d
 \left( {t}^{2}-t \right) ^{3}+ \left( 2\,t-1 \right)  \left( e+f
 \left( {t}^{2}-t \right) +g \left( {t}^{2}-t \right) ^{2}+h \left( {t
}^{2}-t \right) ^{3} \right) .\]
This form is used in \cite{A1} and it gives smaller Gr\"{o}bner bases. It is clear that if $a=b=c=d$ implies that $z=0$ is a center. The composition condition in Lemma 2.3 is satisfied with $s(t)=t^{2}-t$.

The bases for all the cases are given by \\
If $f=g=h=d=c=0$ then $G_{4}=\langle a,b \rangle .$\\
If $g=h=d=c=0$ then $G_{4}=\langle a,b \rangle .$\\
If $ g=h=d=0$ then $G_{5}=\langle a,b,c \rangle .$\\
If $h=d=0$ then $G_{5}=\langle a,b,c \rangle .$\\
If $h=0$ then $G_{10}=\langle a,b,c,d \rangle .$\\
If the degree of $A(t)$ is $7$ then $G_{11}=\langle a,b,c,d \rangle.$

For the case of piecewise linear coefficients, we take
\[ B(t)=2\,t-1 \]
and the forms of $A(t)$ are taken separately. In each case the center follows from Corollary 2.4.
\[A(t)=at+b+\frac{1}{2}\, \left( c-a \right)  \left( t-\frac{1}{2}+ \left| t-\frac{1}{2} \right| \right)\]
The basis is given by
\[G_{4}=\langle a-c,c+2\,b \rangle\]
and
\[A(\frac{1}{2})=\frac{1}{2}\,a+b = 0.\]
With three segments, we take
\[A(t)=at+b+\frac{1}{2}\, \left( c-a \right)  \left( t-\frac{1}{3}+ \left| t-\frac{1}{3} \right| \right)+\frac{1}{2}\, \left( d-c \right) \left (t-\frac{2}{3} +\left| t-\frac{2}{3} \right| \right)\]
And this gives
\[G_{4}=\langle a - d, c + 2 d + 6 b \rangle\]
and
\[ A(\frac{1}{2})= \frac{a}{3} + b + \frac{c}{6}=0\]
Similarly, for $4,5,6$ and $7$ segments the forms of $A(t)$ and the bases are given in by:
\begin{gather*}
A(t)=at+b+\frac{1}{2}\, \left( c-a \right)  \left( t-\frac{1}{4}+ \left| t-\frac{1}{4} \right| \right)+\frac{1}{2}\, \left( d-c \right) \left (t-\frac{2}{4} +\left| t-\frac{2}{4} \right| \right) +\\\frac{1}{2} \,\left( e-d \right) \left(t-\frac{3}{4} +\left|t-\frac{3}{4} \right| \right)
\end{gather*}
\[ G_{5}=\langle a - e, 4 b + d + e, c - d \rangle\]
\[A(\frac{1}{2})= \frac{a}{4} + b + \frac{c}{4}=0 \]
 \begin{gather*}
 A(t)=at+b+\frac{1}{2}\, \left( c-a \right)  \left( t-\frac{1}{5}+ \left| t-\frac{1}{5} \right| \right)+\frac{1}{2}\, \left( d-c \right) \left (t-\frac{2}{5} +\left| t-\frac{2}{5} \right|  \right)+\\\frac{1}{2} \,\left( e-d \right) \left( t-\frac{3}{5} +\left| t-\frac{3}{5} \right|\right)  + \frac{1}{2}\, \left(f-e \right) \left(t-\frac{4}{5} + \left| t-\frac{4}{5} \right| \right)
 \end{gather*}
 \[G_{5}=\langle a - f, 10 b + d + 2 e + 2 f, c - e \rangle\]
 \[A(\frac{1}{2})= \frac{a}{5} + b + \frac{c}{5} + \frac{d}{10}=0\]
     \begin{gather*}
     A(t)=at+b+\frac{1}{2}\, \left( c-a \right)  \left( t-\frac{1}{6}+ \left| t-\frac{1}{6} \right| \right)+\frac{1}{2}\, \left( d-c \right) \left (t-\frac{2}{6} +\left| t-\frac{2}{6} \right|  \right)+\\\frac{1}{2} \,\left( e-d \right) \left( t-\frac{3}{6} +\left| t-\frac{3}{6} \right|\right)  + \frac{1}{2}\, \left(f-e \right) \left(t-\frac{4}{6} + \left| t-\frac{4}{6} \right| \right) + \\\frac{1}{2}\, \left (g-f \right) \left(t-\frac{5}{6} +\left| t-\frac{5}{6} \right| \right)
     \end{gather*}
     \[G_{10}=\langle a - g, 6 b + e + f + g, c - f, d - e \rangle\]
     \[A(\frac{1}{2})= \frac{a}{6} + b + \frac{c}{6} + \frac{d}{6}=0\]
   \begin{gather*}
   A(t)=at+b+\frac{1}{2}\, \left( c-a \right)  \left( t-\frac{1}{7}+ \left| t-\frac{1}{7} \right| \right)+\frac{1}{2}\, \left( d-c \right) \left (t-\frac{2}{7} +\left| t-\frac{2}{7} \right|  \right)\\+\frac{1}{2} \,\left( e-d \right) \left( t-\frac{3}{7} +\left| t-\frac{3}{7} \right|\right)  + \frac{1}{2}\, \left(f-e \right) \left(t-\frac{4}{7} + \left| t-\frac{4}{7} \right| \right) + \\\frac{1}{2}\, \left (g-f \right) \left(t-\frac{5}{7} +\left| t-\frac{5}{7} \right| \right) + \frac{1}{2}\, \left(h-g \right) \left(t-\frac{6}{7}+\left|t-\frac{6}{7} \right| \right)
   \end{gather*}
        \[G_{11}=\langle a - h, 14 b + e + 2 f + 2 g + 2 h, c - g, d - f \rangle\]
             \[A(\frac{1}{2})=\frac{a}{7} + b + \frac{c}{7} + \frac{d}{7} + \frac{e}{14}=0\]
\end{enumerate}
\end{proof}
 \section{Quartic Equations}
 A linear recursive formula for computing the multiplicity can be derived as in Section 2. In this case,
 the functions $V_{k}(t)$ are defined by:
\begin{equation}
V_{1}(t)\equiv 1,\,\,V_{k}(t)=-k \int_{0}^{t}[B(s) V_{k-1}(s)+A(s) V_{k-2}(s)+V_{k-3}(s)] ds,\,k>1.
\end{equation}
\begin{prop}
Suppose that $V_{i}(t)$ are defined by the formula (3.1). The solution $z=0$ of equation (1.4) is of multiplicity $k$ if and only if $V_{i}(1)=0$ for $2 \leq i \leq k-1$ and $V_{k}(1) \neq 0$.
\end{prop}
 Now, we consider the equation (1.4). In this case, the origin can not be a center. It is shown in \cite{A4} that $z=0$ is an isolated periodic solution. For a given class of coefficients the Gr\"{o}bner bases are computed until $G_{k}=\langle 1\rangle$; in this case the set of polynomials in $G_{k}$ do not have a common zero and therefore the maximum possible multiplicity is $k$. It is shown in \cite{A4} that $\mu_{1}(2,2)=8$, and it is shown in \cite{A2} that $\mu_{1}(2,3)=10$. We include proofs of these results also.
\begin{proof}
({\emph Theorem 1.2})
\begin{enumerate}
  \item
Let
  \[ B(t)=a+2\,bt+3\,c{t}^{2},\,A(t)=
d+2\,et+3\,f{t}^{2}\]
The Gr\"{o}bner basis is given by
\[G_{7}=\langle 108\,a-11\,{e}^{2},36\,b+11\,{e}^{2},54\,c-11\,{e}^{2},d+e,11\,{e}^{3
}-3240,f \rangle\]
and
\[\eta_{8}=-{\frac {11552}{626535}}\,{e}^{2}.\]
It follows that $G_{8}=\langle 1 \rangle $.

For the case piecewise linear coefficients, we take
\[ B(t)=at+b+\frac{1}{2}\, \left( c-a \right)  \left( t-\frac{1}{2}+ \left| t-\frac{1}{2} \right|
 \right),\]
\[ A(t)=dt+e+\frac{1}{2}\, \left( f-d \right)  \left( t-\frac{1}{2}+ \left| t-\frac{1}{2} \right|
 \right).\]
The basis $G_{7}$ and $\eta_{8}$ are given by
\[G_{7}= \langle 144\,a+7\,{f}^{2},576\,b-7\,{f}^{2},144\,c-7\,{f}^{2},-f+d,f+2\,e,-
27648+7\,{f}^{3} \rangle \]
\[ \eta_{8}=-{\frac {2041}{498960}}\,{f}^{2}\]
These imply that $G_{8}=\langle 1 \rangle $.
  \item
Let
  \[B(t)= a+2\,bt+3\,c{t}^{2},\, A(t)=
d+2\,et+3\,f{t}^{2}+4\,g{t}^{3}.\]
If the multiplicity is greater than $5$ then $\eta_{5}=0$, where
\[\eta_{5}={\frac {1}{1764}}\, \left( gc+210 \right)  \left( 2\,b+3\,c \right).\]
If $ gc+210=0$ then $G_{5}=\langle 1 \rangle $, and
if $2b+3c=0$ then $ G_{10}= \langle 1 \rangle $.\\
For piecewise linear coefficients, let
\[ B(t)=at+b+\frac{1}{2}\, \left( c-a \right)  \left( t-\frac{1}{2}+ \left| t-\frac{1}{2} \right|
 \right) \]
\[A(t)=dt+e+\frac{1}{2}\, \left( f-d \right)  \left( t-\frac{1}{3}+ \left| t-\frac{1}{3} \right|
 \right) +\frac{1}{2}\, \left( g-f \right)  \left( t-\frac{2}{3}+ \left| t-\frac{2}{3}
 \right|  \right).\]
The case that $c-a=0$,
 \[ G_{4}= \langle 6\,eb+fb+2\,db+81,c+2\,b,3\,f+18\,e+5\,d+g,2\,b+a \rangle \]
\[ G_{5}= \langle 1 \rangle. \]
When $c-a \neq 0$, we make the change of variables
\[z \mapsto \frac{1}{c-a}\,z. \]
With this transformation, the equation is of the form
\[ \dot{z}=k\,z^{4}+A(t)\,z^{3}+B(t)\,z^{2},\]
with a non-zero constant $k$. In fact, $k=\frac{1}{(c-a)^{3}}$.
Moreover, the coefficient $c-a$ in $B(t)$ is replaced by $1$. The basis, with $c-a=1$, is given by
\begin{gather*}
G_{9}= \langle 66735388183208154600960\,{f}^{4}+5829122567397869818848\,{f}^{3}-
68783721774316079552\,fe\\+139990051412348601632\,{f}^{2}+
1344723268054007200\,e+146583972817817393\,f,\\17101027722240\,e{f}^{2}-
5269542106752\,{f}^{3}+\\733610765888\,fe-233959770968\,{f}^{2}-
6125396800\,e+1693244383\,f,8\,eb-e,\\82432\,{e}^{2}-17024\,fe-1904\,{f}
^{2}+920\,e-269\,f,8\,fb-f,\\2\,g+6\,e+f,4\,a+8\,b+5,10368\,k+32\,e-11\,
f,f+6\,e+2\,d \rangle
\end{gather*}
\begin{gather*}
\eta_{10}=-{\frac {104057406529615499}{780994714281201717844377600}}\,{f}^{3}-\\{
\frac {867941774841100820209033}{96522790757355293736025622524723200}}
\,fe+\\{\frac {1666102537252531452515843}{
140396786556153154525128178217779200}}\,{f}^{2}+{\frac {
2196152718747243263819885}{206810570805324733709188812087754752}}\,e-\\{
\frac {2900881012686053802785600993}{
951328625704493775062268535603671859200}}\,f.
\end{gather*}
\[ G_{10}= \langle g,4\,a+8\,b+5,k,d,e,f \rangle \]
\item
If $B(t)$ is a linear function and the multiplicity is greater than $2$, then $B(t)=a-2\,a\,t$. Again with the change of variables
\[z\mapsto \frac{1}{a}\,z\]
We consider the differential equation
\[ \dot{z}=k \,z^{4}+A(t)\,z^{3}+B(t)\,z^{2},\]
with
  \[ B(t)= 2\,t-1,\, A(t)=
b+ct+d{t}^{2}+e{t}^{3}+f{t}^{4}+g{t}^{5}.\]
Here, $k=\frac{1}{a^{3}}$. The bases are given by
\begin{gather*}
G_{10}= \langle
37511692566915157189513021571250141246445562631168000\,a+\\
76954078025109319792419383346327250\,{g}^{4}f+\\
6849063961008228076202649414049259748983323392\,fg+\\
192385195062773299481048458365818125\,{g}^{5}+\\
17122659902520570190506623535123149372458308480\,{g}^{2},\\
812753338949828405772782134043753060339653857008640\,b-\\
3039686081991818131800565642179926375\,{g}^{4}+\\
6396626065376256805845080981938604767430030478336\,g,\\
12699270921091068840199720844433641567807091515760\,c+\\
2539854184218213768039944168886728313561418303152\,f+\\
423247429138101258858306608404799875\,{g}^{4}+\\
4341516136354618897336295790080428452939652335036\,g,\\
496186409615279857004140496974208217545576225280\,d-\\
595423691538335828404968596369049861054691470336\,f-\\
38477039012554659896209691673163625\,{g}^{4}-\\
1028512938658225137299402571203191057331286500736\,g,\\
744279614422919785506210745461312326318364337920\,e+\\
1488559228845839571012421490922624652636728675840\,f+\\
38477039012554659896209691673163625\,{g}^{4}+\\
2020885757888784851307683565151607492422438951296\,g,\\
455446932871017250994724610810483152793296813865452398540570112\,{f}^{
2}+\\2277234664355086254973623054052415763966484069327261992702850560\,f
g-\\49539000048633777493317082388741947622188883125\,{g}^{5}+\\
2845280986013026886327275526892991130747287056580695907322835840\,{g}^
{2},\\3984808136948929447103185415887785672672861696\,{g}^{3}-\\
14218838340357434859762303459719245720157600157874298880000+\\
38477039012554659896209691673163625\,{g}^{6} \rangle
\end{gather*}
When $g=0$,
\[ G_{9}= \langle k,b+c-e,d+2\,e,f \rangle \]
When $f=g=0$,
\[ G_{5}= \langle k,b+c-e,d+2\,e \rangle \]
When $e=f=g=0$,
\[ G_{5}= \langle k,b+c,d \rangle \]
In the corresponding piecewise linear coefficients, we take $B(t)=2\,t-1$ and
 \[ A(t)=at+b+\frac{1}{2}\, \left( c-a \right)  \left( t-\frac{1}{2}+ \left| t-\frac{1}{2} \right|
 \right) \]
\[ G_{4}= \langle 144+ ad+ 2\,ac,8\,c+d+3\,b \rangle, \eta_{5}=\frac{a}{6}. \]
When
\[A(t)=dt+e+\frac{1}{2}\, \left( f-d \right)  \left( t-\frac{1}{3}+ \left| t-\frac{1}{3} \right|
 \right) +\frac{1}{2}\, \left( g-f \right)  \left( t-\frac{2}{3}+ \left| t-\frac{2}{3}
 \right|  \right),\]
we have
\[ G_{4}= \langle ad+6\,ac+2\,ab-81,18\,c+3\,d+5\,b+e \rangle, \eta_{5}=\frac{2\,a}{9}. \]
When
\begin{gather*}
A(t)=at+b+\frac{1}{2}\, \left( c-a \right)  \left( t-\frac{1}{4}+ \left| t-\frac{1}{4} \right| \right)+\frac{1}{2}\, \left( d-c \right) \left (t-\frac{2}{4} +\left| t-\frac{2}{4} \right| \right) +\\\frac{1}{2} \,\left( e-d \right) \left(t-\frac{3}{4} +\left|t-\frac{3}{4} \right| \right),
\end{gather*}
we have
\begin{gather*}
 G_{8}= \langle 24301478794941\,a{d}^{2}-5170471968081920\,c-6997968367619776\,d,\\
348941857826215204551\,{d}^{3}-500251867306017904135800\,ad+\\
62937496133967787727964160,86251\,{a}^{2}-122304\,c-7280\,d,690\,ac-\\
519\,ad+125456,61144830207333600\,{c}^{2}-6622547301720987\,{d}^{2}+\\
3643254234403888940\,a,-302960854428691360\,a-280441397922231\,{d}^{2}
+\\3057241510366680\,dc,7841\,e-19320\,c+6691\,d,23523\,b+115712\,c+\\7261
\,d,116632\,c+23523\,f+6569\,d \rangle
\end{gather*}
and
\[ \eta_{9}={\frac {2426539331050271747}{36743835937950792000}}\,d-{\frac {
322603494879515897}{3674383593795079200}}\,c .\]
\[ G_{9}= \langle 1 \rangle \]
Finally, for the case
\begin{gather*}
 A(t)=at+b+\frac{1}{2}\, \left( c-a \right)  \left( t-\frac{1}{5}+ \left| t-\frac{1}{5} \right| \right)+\frac{1}{2}\, \left( d-c \right) \left (t-\frac{2}{5} +\left| t-\frac{2}{5} \right|  \right)+\\\frac{1}{2} \,\left( e-d \right) \left( t-\frac{3}{5} +\left| t-\frac{3}{5} \right|\right)  + \frac{1}{2}\, \left(f-e \right) \left(t-\frac{4}{5} + \left| t-\frac{4}{5} \right| \right)
 \end{gather*}
we make the change of variable $z \mapsto \frac{1}{a} \,z$ and let $B(t)=2\,t-1$. The basis is given by
\[ G_{10}=\langle k,e+10\,c+2\,d+2\,g,b-g,-d+f \rangle \]
\end{enumerate}
\end{proof}
\begin{rema}
 \begin{enumerate}
 \item
 To see the effect of connection points on multiplicity, we repeat the computations in Theorem 1.2(i) with many values of the connection points. The results are the same. In particular, $\mu_{2}92,2)=8$. As an example, we give the Gr\"{o}bner basis $G_{7}$ and  $\eta_{8}$ when the connection point is at $\frac{3}{7}$.
\[G_{7}=\langle81486729\,a+5324000\,{f}^{2},190135701\,b-3327500\,{f}^{2},6036054\,c
-166375\,{f}^{2},\]
\[9027\,d-6836\,f,21063\,e+8810\,f,2013137500\,{f}^{3}-
9344599297047\rangle,\]
\[\eta_{8}=-{\frac {652948208}{152171939367}}\,{f}^{2}.\]
\item
It follows from the proof of Theorem 1.2(i) that if $z=0$ has the maximum multiplicity then $A(t)$ is a linear function.
 The values of $b$ and $c$ when the multiplicity is $8$ imply that the parabola representing $B(t)$ has a vertex at $t=\frac{1}{2}$. The other similarities between polynomial coefficients and piecewise linear coefficients are: there is only one equation with multiplicity $8$, and when the multiplicity is $8$ the origin is unstable.
 \end{enumerate}
\end{rema}
Finally, we consider a case where $A(t)$ has two segments connected at any point inside the interval $[0,1]$.
\begin{theo}
 Consider the class of equations in which $B(t)$ is a linear function and $A(t)$ is a piecewise linear with two segments connected at a point $h$, with $0 \leq h \leq 1$. Then, $\mu_{2}(1,2)=5$.
\end{theo}
\begin{proof}
We put $\eta_{2}=0$, and then take
\[B(t)=2\,at-a.\]
The function $A(t)$ is written as $A(t)=b\,t+c$ when $0 \leq t \leq h$, and $A(t)=d\,t+b\,h+c-d\,h$ when $h \leq t \leq 1$.
To compute $\eta_{3},\eta_{4},\eta_{5}$, we use the following formulae derived in \cite{A4}.
\[ \eta_{3}=\int_{0}^{1} A(t) \,dt,\]
\[ \eta_{4}=\int_{0}^{1} [A(t)\, \bar{B}(t)+1]\,dt,\]
\[ \eta_{5}=\int_{0}^{1} [A(t)\,(\bar{B}(t))^{2}+2\,\bar{B}(t)] \,dt,\]
where $\bar{B}(t)=\int_{0}^{t} B(s)\,ds$. The Gr\"{o}bner basis $G_{5}$ is given by
\[G_{5}=\langle ab-ad-108,36\,{c}^{2}+6\,cd+30\,bc+bd+7\,{b}^{2}+{d}^{2},2\,ac+108\,h
+ad+36,\]
\[3\,bh-3\,dh+b+6\,c+2\,d,-5\,d-24\,c+9\,dh+18\,ch-7\,b,3\,{h}^{2
}+1-3\,h \rangle.\]
It is clear that the last polynomial does not have a real solution. Hence, $\mu_{2}(1,2)=5$.
\end{proof}

\bibliographystyle{amsplain}

\end{document}